\documentclass[reqno]{amsart}
\usepackage{}
\usepackage{mathrsfs}
\usepackage{color}
\usepackage{amsmath}
\usepackage{amsfonts}
\usepackage{amssymb}
\usepackage{graphicx}%
\usepackage{hyperref}

 \newtheorem{Theorem}{Theorem}[section]
 
 \newtheorem{Lemma}[Theorem]{Lemma}

 \newtheorem{Question}[Theorem]{Question}

 \newtheorem{Remark}[Theorem]{Remark}

 \numberwithin{equation}{section}
\begin{document}

\title[decreasing equisingular approximations with logarithmic poles]
 {Nonexistence of decreasing equisingular approximations with logarithmic poles}

\author{Qi'an Guan}
\address{Qi'an Guan: Beijing International Center for Mathematical Research, and School of Mathematical Sciences,
Peking University, Beijing, 100871, China.}
\email{guanqian@amss.ac.cn}

\thanks{The author was partially supported by NSFC}



\date{\today}

\dedicatory{}

\commby{}

%%% ----------------------------------------------------------------------

\begin{abstract}
In this article, we present that for any complex manifold whose dimension is bigger than one,
there exists a multiplier ideal sheaf such that
there don't exist equisingular weights with logarithmic poles,
which are not smaller than the orginal weight.
A direct consequence is the nonexistence of decreasing equisingular approximations with logarithmic poles.
\end{abstract}

%%% ----------------------------------------------------------------------
\maketitle
%%% ----------------------------------------------------------------------

\section{Introduction}\label{sec:1}

Let $\varphi$ be a plurisubharmonic function (see \cite{kisel}) on a complex manifold $X$.
Following Nadel \cite{Nadel90},
one can define the multiplier ideal sheaf $\mathcal{I}(\varphi)$ (with weight $\varphi$) to
be the sheaf of germs of holomorphic functions $f$ such that
$|f|^{2}e^{-2\varphi}$ is locally integrable (see also \cite{siu05},
\cite{siu09}, \cite{demailly2010}, \cite{demailly-book}, etc.).

In \cite{demailly-note2000} (see also \cite{demailly2010}),
Demailly shows that for any given quasi-plurisubharmonic function $\varphi$
(i.e., locally can be expressed by $\psi+v$, where $\psi$ is plurisubharmonic function and $v$ is smooth) on compact Hermitian manifold $M$,
there exist quasi-plurisubharmonic functions $\varphi_{S,j}$ $(j=1,2,\cdots)$ on $M$ with smooth poles
satisfying
$$\mathcal{I}(\varphi)=\mathcal{I}(\varphi_{S,j})$$
$(j=1,2,\cdots)$ ("equisingularity"),
which are decreasing convergent to $\varphi$, when $j$ goes to $\infty$.

It is called that a quasi-plurisubharmonic function $\varphi_{A}$ has logarithmic poles if
there exist holomorphic functions $g_{k}$ $(k=1,\cdots,N)$ such that
$$\varphi_{A}=c\log\sum_{k=1}^{N}|g_{k}|^{2}+O(1),$$
where $c\in\mathbb{R}$ (see \cite{demailly-note2000},\cite{demailly2010}).
In \cite{demailly-note2000} (see also \cite{demailly2010}),
Demailly asked
\begin{Question}
\label{Ques:Demailly}
For any given quasi-plurisubharmonic function $\varphi$ on $M$,
can one choose equisingular quasi-plurisuhbarmonic functions $\varphi_{A,j}$ $(j=1,2,\cdots)$ on $M$ with logarithmic poles,
which are decreasing convergent to $\varphi$ $(j\to\infty)$?
\end{Question}

In this article,
we give negative answers to Question \ref{Ques:Demailly} for any dimension $n\geq 2$ by the following theorem
\begin{Theorem}
\label{thm:analytic_Guan}
For any complex manifold $M$ (compact or noncompact) $dim M\geq 2$ and $z_{0}\in M$,
there exists a quasi-plurisubharmonic function $\varphi$ on $M$ such that
for any plurisubharmonic function $\varphi_{A}\geq \varphi$ near $z_{0}\in M$ with logarithmic poles,
\begin{equation}
\label{equ:first_6012}
c_{z_{0}}(\varphi)<c_{z_{0}}(\varphi_{A})
\end{equation}
holds,
where $c_{z_{0}}(\varphi):=\sup\{c|\mathcal{I}(c\varphi)_{z_{0}}=\mathcal{O}_{z_{0}}\}$ is the complex singularity exponent of $\varphi$.
\end{Theorem}

We prove Theorem \ref{thm:analytic_Guan} by considering the following
\begin{Remark}
\label{rem:guan0620}
Let
$$\varphi_{1}:=\log(\max\{|z_{1}|,\cdots,|z_{n-1}|,|z_{n}|^{a}\}),$$
where $a\in(1,1+a/2)$ is a irrational number, and $(z_{1},\cdots,z_{n})$ are coordinates on $\mathbb{C}^{n}$.
Let
$$\varphi_{2}:=\max\{\varphi_{1}-18n,6\log(|z_{1}|^{2}+\cdots+|z_{n}|^{2})-6n\}.$$
Let
$$\varphi:=-M_{\eta}(-\varphi_{2},0),$$
where $M_{\eta}(t_{1},t_{2})$ is in Lemma (5.18) in \cite{demailly-book},
which satisfying

(1) $M_{\eta}(t_{1},t_{2})$ is smooth on $\mathbb{R}^{2}$;

(2) $M_{\eta}(t_{1},t_{2})|_{\{t_{2}+2\varepsilon_{0}\leq t_{1}\}}=t_{1}$ and $M_{\eta}(t_{1},t_{2})|_{\{t_{1}+2\varepsilon_{0}\leq t_{2}\}}=t_{2}$,
$\\$
and $\eta:=(\varepsilon_{0},\varepsilon_{0})$, $\varepsilon_{0}=\frac{1}{1000}$.
\end{Remark}

In following two remarks present that $\varphi$ in Remark \ref{rem:guan0620}
is quasi-plurisubharmonic, which can be can be extended to $M$.

Let

(1) $A_{1}:=\{z|\log(\max_{j=1,\cdots,n}|z_{j}|)<0\}$;

(2) $A_{2}:=\{6\log(|z_{1}|^{2}+\cdots+|z_{n}|^{2})-6n<-2\varepsilon_{0}\}$;

(3) $A_{3}:=\{z|\log(\max_{j=1,\cdots,n}|z_{j}|)<-6\log n\}$.

It is clear that
$$A_{3}\subset\subset A_{1}\subset\subset A_{2}.$$
The following remark shows that $\varphi$ in Remark \ref{rem:guan0620} is quasi-plurisubharmonic.
\begin{Remark}
\label{rem:guan0620a}
As
\begin{equation}
\begin{split}
6\log(|z_{1}|^{2}+\cdots+|z_{n}|^{2})-6n
&\geq12\log(\max_{j=1,\cdots,n}|z_{j}|)-6n
\\&\geq a\log(\max_{j=1,\cdots,n}|z_{j}|)-6n\geq \varphi_{1}
\end{split}
\end{equation}
on $A_{1}^{c}$,
then
\begin{equation}
\label{equ:first_0620a}
\varphi_{2}(z)|_{A_{1}^{c}}=6\log(|z_{1}|^{2}+\cdots+|z_{n}|^{2})-6n.
\end{equation}
By (1) in Remark \ref{rem:guan0620},
it follows that
$\varphi$ is smooth on
$(A_{1}^{c})^{o}$.

As 
$$\varphi_{1}|_{A_{2}}<6\log(|z_{1}|^{2}+\cdots+|z_{n}|^{2})-18n<6n-2\varepsilon_{0}-18n<-2\varepsilon_{0}$$
and
$$(6\log(|z_{1}|^{2}+\cdots+|z_{n}|^{2})-6n)|_{A_{2}}<-2\varepsilon_{0},$$
then it follows that $\varphi_{2}|_{A_{2}}<-2\varepsilon_{0}$.
By using (2) in Remark \ref{rem:guan0620},
it follows that $\varphi|_{A_{2}}=\varphi_{2}$ is plurisubharmoic on
$A_{2}$.

Note that
$$(A_{1}^{c})^{o}\cup A_{2}=\mathbb{C}^{n}.$$
Then $\varphi$ in Remark \ref{rem:guan0620} is quasi-plurisubharmonic.
\end{Remark}

The following remark shows that $\varphi$ in Remark \ref{rem:guan0620} can be extended to $M$.
\begin{Remark}
\label{rem:guan0620b}
By equality \ref{equ:first_0620a} and
and (2) in Remark \ref{rem:guan0620},
then it is clear that
\begin{equation}
\begin{split}
&\varphi|_{\{6\log(|z_{1}|^{2}+\cdots+|z_{n}|^{2})-6n>2\varepsilon_{0}\}}
\\&=-M_{\eta}(-6\log(|z_{1}|^{2}+\cdots+|z_{n}|^{2})+6n,0)|_{\{6\log(|z_{1}|^{2}+\cdots+|z_{n}|^{2})-6n>2\varepsilon_{0}\}}\equiv0.
\end{split}
\end{equation}
\end{Remark}

The following remark present the singularity of $\varphi$ in Remark \ref{rem:guan0620}
\begin{Remark}
\label{rem:guan0620c}
As
\begin{equation}
\begin{split}
6\log(|z_{1}|^{2}+\cdots+|z_{n}|^{2})-6n
&\leq12\log(\max_{j=1,\cdots,n}|z_{j}|)+6\log n-6n
\\&\leq a\log(\max_{j=1,\cdots,n}|z_{j}|)-6n\leq \varphi_{1}
\end{split}
\end{equation}
on $A_{3}$,
then
$$\varphi_{2}|_{A_{3}}=\varphi_{1}.$$
By Remark \ref{rem:guan0620a} ($\varphi|_{A_{2}}=\varphi_{2}$) and $A_{3}\subset\subset A_{1}$,
it follows that
$$\varphi|_{A_{3}}=\varphi_{1}.$$
\end{Remark}

Using Theorem \ref{thm:analytic_Guan}, we answer Question \ref{Ques:Demailly} by contradiction
\begin{Remark}
\label{rem:first_0613}
If not,
then
for the plurisubharmonic function $\varphi_{1}=\varphi|_{A_{3}}$ in Remark \ref{rem:guan0620},
there exists a plurisubharmonic function $\varphi_{A}$ with logarithmic poles near $o$ satisfying
$c_{o}(\varphi_{1})\varphi_{A}\geq c_{o}(\varphi_{1})\varphi_{1}$,
such that $e^{-2c_{o}(\varphi_{1})\varphi_{A}}$
is not integrable near $o$.
By Berndtsson's solution of the openness conjecture (\cite{berndtsson13}) posed by Demailly and Kollar (\cite{D-K01}),
it follows that
$c_{o}(\varphi_{A})\leq c_{o}(\varphi_{1})$,
which contradicts Theorem \ref{thm:analytic_Guan}.
\end{Remark}

\section{Some Preparations}

In this section, we recall some known results and present some observations.

\subsection{A sharp lower bound for the log canonical threshold for dimension $2$ case}

In \cite{D-H2014},
Demailly and Hiep present the following
\begin{Theorem}
\label{thm:D-H}(\cite{D-H2014})
Let $\varphi_{A}\geq \varphi_{1}$ be a plurisubharmonic function near $o\in\mathbb{C}^{2}$ with logarithmic poles,
then
\begin{equation}
\label{equ:first_0612b}
c_{o}(\varphi_{A})\geq \frac{1}{e_{1}(\varphi_{A})}+\cdots+\frac{e_{n-1}(\varphi_{A})}{e_{n}(\varphi_{A})},
\end{equation}
where $e_{k}(\varphi_{A}):=\nu((dd^{c}\varphi_{A})^{k},o)$ ($e_{1}(\varphi_{A})=\nu(\varphi_{A},o)$).
\end{Theorem}

As $\varphi_{A}\geq \varphi$,
then one can obtain
\begin{equation}
\label{equ:first_0612d}
e_{n}(\varphi_{A})\leq e_{n}(\varphi)=a
\end{equation}
and
\begin{equation}
\label{equ:0620a}
e_{k}(\varphi_{A})\leq e_{k}(\varphi)=1\,\,(k\in\{1,\cdots,n-1\})
\end{equation}
(by using Second comparison theorem (7.8) and Example (6.11) in chapter III of \cite{demailly-book})

\subsection{Observations}
Note that $c_{o}(\log\sum_{k=1}^{N}|g_{k}|^{2})$ is a rational number (see \cite{D-K01}),
and the Lelong number $\nu(\log\sum_{k=1}^{N}|g_{k}|^{2},o)$ is a integer (see \cite{demailly-book}),
where $g_{k}$ are holomorphic functions near $o\in\mathbb{C}^{n}$.
Then it is clear that
\begin{Lemma}
\label{lem:first_0613a}
Let plurisubharmonic function $\varphi_{A}:=c\log\sum_{k=1}^{N}|g_{k}|^{2}+O(1)$ near $o$,
where $c\in\mathbb{R}^{+}$, and $g_{k}$ are holomorphic functions near $o$.
Then
$$c_{o}(\varphi_{A})\nu(\varphi_{A},o)=c_{o}(\log\sum_{k=1}^{N}|g_{k}|^{2})\nu(\log\sum_{k=1}^{N}|g_{k}|^{2},o)$$
is a rational number.
\end{Lemma}

We prove Theorem \ref{thm:analytic_Guan} by using the following lemma:

\begin{Lemma}
\label{lem:first_0612}
Let $\varphi_{A}\geq \varphi_{1}$ (as in Remark \ref{rem:guan0620}) be a plurisubharmonic function near $o\in\mathbb{C}^{n}$ with logarithmic poles,
where $a>1$ is an irrational number.
Assume that $c_{o}(\varphi_{A})=c_{o}(\varphi_{1})(=n-1+\frac{1}{a})$ ($c_{o}(\varphi_{1})=n-1+\frac{1}{a}$ see \cite{D-K01}).
Then $\nu(\varphi_{A},o)<\nu(\varphi_{1},o)(=1).$
\end{Lemma}

\begin{proof}
As $\varphi_{A}\geq \varphi_{1}$,
then it is clear that $\nu(\varphi_{A},o)\leq\nu(\varphi_{1},o)$.

We prove Lemma \ref{lem:first_0612} by contradiction:
if not, then
$\nu(\varphi_{A},o)=\nu(\varphi_{1},o)(=1)$.
By Lemma \ref{lem:first_0613a},
it follows that
$c_{o}(\varphi_{A})\nu(\varphi_{A},o)$ is a rational number,
which contradicts
$\nu(\varphi_{A},o)c_{o}(\varphi_{A})=1(n-1+\frac{1}{a})=n-1+\frac{1}{a}.$
\end{proof}

\section{Proof of Theorem \ref{thm:analytic_Guan}}
We prove Theorem \ref{thm:analytic_Guan} by contradiction:
if not, then there exists a plurisubharmonic function
$\varphi_{A}\geq \varphi_{1}$ near $o$ with logarithmic poles
such that
\begin{equation}
\label{equ:first_0615a}
c_{o}(\varphi_{1})=c_{o}(\varphi_{A})
\end{equation}
($\varphi_{A}\geq\varphi_{1}\Rightarrow c_{o}(\varphi)\leq c_{o}(\varphi_{A})$).

By inequalities \ref{equ:first_0612b} and \ref{equ:first_0612d}
it follows that
\begin{equation}
\label{equ:first_0612e}
\begin{split}
c_{o}(\varphi_{A})
&\geq\frac{1}{e_{1}(\varphi_{A})}+\cdots+\frac{e_{n-2}(\varphi_{A})}{e_{n-1}(\varphi_{A})}+\frac{e_{n-1}(\varphi_{A})}{e_{n}(\varphi_{A})}
\\&\geq
\frac{n-1}{e_{n-1}^{\frac{1}{n-1}}(\varphi_{A})}+\frac{e_{n-1}(\varphi_{A})}{e_{n}(\varphi_{A})}
\\&\geq
\frac{n-1}{e_{n-1}^{\frac{1}{n-1}}(\varphi_{A})}+\frac{e_{n-1}(\varphi_{A})}{a}.
\end{split}
\end{equation}

Note that function $f(t):=\frac{n-1}{t^{\frac{1}{n-1}}}+\frac{t}{a}$ $(t\in(0,a^{\frac{n-1}{n}}{\,}])$
is strictly decreasing with respect to $t$.
If $e_{n-1}(\varphi_{A})\leq1$,
then we have
\begin{equation}
\label{equ:first_0612f}
\begin{split}
\frac{n-1}{e_{n-1}^{\frac{1}{n-1}}(\varphi_{A})}+\frac{e_{n-1}(\varphi_{A})}{a}
\geq
n-1+\frac{1}{a}=c_{o}(\varphi),
\end{split}
\end{equation}
moreover $"="$ in inequality \ref{equ:first_0612f} holds if and only if $e_{n-1}(\varphi_{A})=1$.
If $e_{n-1}(\varphi_{A})<1$, then it follows that $c(\varphi_{A})>n-1+\frac{1}{a}$
(by inequality \ref{equ:first_0612e}),
which contradicts equality \ref{equ:first_0615a}.
Then it suffices to consider the case $e_{n-1}(\varphi_{A})=1$.

Note that the second $"\geq"$ of inequality \ref{equ:first_0612e} is $"="$ if and only if
$e_{1}(\varphi_{A})=\cdots=e_{n-1}(\varphi_{A})=1$ (by $e_{n-1}(\varphi_{A})=1$).
By Lemma \ref{lem:first_0612},
it follows that $e_{1}(\varphi_{A})<1$,
which implies that
the second $"\geq"$ of inequality \ref{equ:first_0612e} is $">"$.
Using inequality \ref{equ:first_0612f},
we obtain that
$$c_{o}(\varphi_{A})>n-1+\frac{1}{a},$$
which contradicts
equality \ref{equ:first_0615a}.

Then Theorem \ref{thm:analytic_Guan} has been proved.

\vspace{.1in} {\em Acknowledgements}. The author would like to thank Professor Xiangyu Zhou for helpful discussions,
and reminding the author to consider Question 1.1 on compact Hermitian manifolds.

\bibliographystyle{references}
\bibliography{xbib}

\end{document}